\documentclass{article}
\usepackage{amssymb}
\usepackage{amsfonts}
\usepackage{amsmath}

\newtheorem{theorem}{Theorem}

\newtheorem{corollary}[theorem]{Corollary}

\newtheorem{definition}[theorem]{Definition}

\newtheorem{lemma}[theorem]{Lemma}

\newenvironment{proof}[1][Proof]{\noindent\textbf{#1.} }{\ \rule{0.5em}{0.5em}}

\begin{document}

\title{CM Method and Expansion of Numbers}
\author{A Abdurrahman\thanks{%
Ababdu@ship.edu} \\
Department of Physics\\
Shippensburg University of Pennsylvania\\
1871 Old Main Drive\\
Shippensburg, PA 17257\\
USA}
\maketitle

\begin{abstract}
We show that an iterative method for computing the center of mass (CM) of $q$
units of mass, placed on a unit interval $\left[ 0,1\right] $ along the $x$%
-axis, give rise to a simple procedure for expanding rational numbers less
than unity in powers of \ $r/s<1$, with $r,s$, integers larger than $0$. The
method is then extended to all numbers, real or complex, though the
procedure for none rational numbers is more time consuming. We also show how
our method provides a natural way to generalize Jacobsthal numbers.
Moreover, the method provides a way to generate infinitely many sequences of
numbers, of which many play an important rule in mathematical sciences and
engineering, to name few, Jacobsthal sequence, Fibonacci sequence, and Pell
sequence.
\end{abstract}

\section{A general outline of the method}

In this paper we are going to give a method based on the concept of center
of mass in solid mechanics which will allow us to expand any rational number 
$\left( p/q\right) $ with $q>p>0$, as an infinite series in powers of $%
\left( 1/2\right) $. The method can also be used with some modification to
express any rational number $\left( p/q\right) $ with $q>p>0$ as an infinite
series in powers of $\left( r/s\right) $ with $r,s$ (integers) and $s>r>0$ .
The method basically involves dividing $q$ units of mass placed on the unit
interval $\left[ 0,1\right] $ of the real $x$-axis, into two groups $p$ and $%
q-p$, with $q-p>0$\footnote{%
In what follows both the mass and distance are considered to be
dimensionless. Moreover the masses are considered to be point like. The
assumption that $q-p>0$ is an "integer" may be relaxed as we shall see later.%
}. We start by placing $q-p$ units of mass at the origin $x=0$ and $p$ units
of mass at $x=1$. We call this arrangement the initial configuration or
state. Without lose of generality, we take the initial position of the
larger collection, that is, of $q-p$ to be the zero approximation of the
center of mass of the given arrangement, which shall call $x_{cm}^{\left(
p/q\right) }\left( 0\right) $\footnote{%
The initial choice for $x_{cm}\left( 0\right) $ is immaterial, any point on
the $x$-axis will do, for any arbitrary choice of $x_{cm}\left( 0\right) $,
the method converges to the same limit. The reader should note that calling $%
x_{cm}\left( 0\right) $ the zero approximation may not correspond to the
precise definition of approximation in the literature.}. The procedure is
then to consider the $p$ units of mass and an equal number $p^{\prime }$
units of mass from the larger group $q-p$ and compute their center of mass%
\footnote{%
In fact $p^{\prime }$ need not equal $p$ as we shall see later. But this
choice is the simplest and it serves to illustrates the method with out loss
of generality.}. By symmetry their center of mass is the midpoint of our
interval, $x=1/2$. This provides the first correction to $x_{cm}^{\left(
p/q\right) }\left( 0\right) $ in our method and we have $x_{cm}^{\left(
p/q\right) }\left( 1\right) =x_{cm}^{\left( p/q\right) }\left( 0\right)
+1/2=0+1/2$, where now $x_{cm}^{\left( p/q\right) }\left( n\right) $ is the $%
n$th approximation of the center of mass of the $q$ units of mass in their
initial configuration. Next we place the $p+p^{\prime }$ units of mass at $%
x_{cm}^{\left( p/q\right) }\left( 1\right) $, i.e., $x=1/2$, and leave the
reaming units of mass $q-\left( p+p^{\prime }\right) $ at $x=0$. Now we
repeat the same procedure on the new interval along the $x$-axis, $\left[
0,1/2\right] $, i.e., the interval where the new arrangement of the $q$
units of mass now reside. To obtain the second order approximation of the
center of mass for the $q$ units of mass compute the center of mass for $%
p+p^{\prime }$ and an equal amount of units of mass taken from $%
q-p-p^{\prime }$ , which is simply $\left( 1/2\right) \left( 1/2\right) $.
This value is the second order correction to the center of mass. Observe in
this case the center of mass shifts to the left while for the first order
correction shifts to the right. Hence, $x_{cm}^{\left( p/q\right) }\left(
2\right) =x_{cm}^{\left( p/q\right) }\left( 1\right) -\left( 1/2\right)
\left( 1/2\right) =0+\left( 1/2\right) -\left( 1/2\right) \left( 1/2\right) $%
. Placing the $2p^{\prime }+2p$ units of mass at $x_{cm}^{\left( p/q\right)
}\left( 2\right) $ and repeating the same steps indefinitely\footnote{%
When $p$ becomes larger than $q-p$, the procedure is reversed. That is we
are always taking the smaller group of masses and combing it with an equal
number from the larger group. For $q=even$, sometimes this process
terminates and produces a finite number of terms. For example for $p=1,q=4$,
the expansion of $1/4$ in powers of $\left( 1/2\right) $ terminates at the
second iteration. However, the expansion of $\left( 1/4\right) $ in powers
of $\left( 1/3\right) $ gives an infinite series in powers of $\left(
1/3\right) $.}, we see that the center of mass converges to a limit point%
\footnote{%
We shall call the the initial choice of the ratio of $p$ to $p^{\prime }$ as
the iterative weight and we simply denote it by $w=p:p^{\prime }$. In the
present case, $w=1:1$. We shall see later that other choices of $w$ lead to
expansions in powers different from $1/2$.}. Thus with suitable choice of
the integers $q$ and $p$, this iterative method results in an infinite
series that can be identified with any rational number less than unity whose
value is the limit of $x_{cm}^{\left( p/q\right) }\left( n\right) $ as $n$
(the number of iterations) tends to infinity. In the following sections, we
will illustrate our method through some concrete examples, relax the
assumption that $p/q$ is a rational number smaller than $unity$ and extend
the method to all numbers real or complex.

\section{Expansion of rational numbers smaller than unity}

To make notation more precise, let $X^{\left( p/q,r/s\right)
}=x_{cm}^{\left( p/q,r/s\right) }$, where $q-p>0$ units of mass are placed
at the origin, $p$ units of mass are placed at one unit length from the
origin and center of mass is expanded in powers of $\left( r/s\right) $. We
now consider the expansion of rational numbers $\left( p/q\right) $ with $%
q>p>0$. In fact it suffice to consider only the rational number $\left(
1/q\right) $ since $p/q=p\left( 1/q\right) $ and $p$ can always be absorbed
in the coefficients of the expansion. Let us consider a system of $q$ units
of point-like masses, with $q=2,3,4,5,...\,$. Let $p$ units of mass be
located at one unit of distance away from the origin along the positive real
axis, and let the remaining $\left( q-p\right) $ units of mass be located at
the origin. \ It should be noted here that all quantities are assumed to be
dimensionless.

To illustrate our method we could just focus on a specific rational number
and expand it in powers\footnote{%
One can expand in powers of $r/s=1/3$ or any other rational number less than
one.} of $r/s=1/2$. The simplest nontrivial case is $p/q=1/3$. With the
choice $p=1$ and $q=3$, the center of mass $\left( x_{cm}^{\left( 1/3\right)
}\right) $ for this combination follows at once by symmetry or from the
simple formula, for point masses located on the $x$-axis, derived in
mechanics\footnote{%
See any book on solid mechanics.} 
\begin{equation}
x_{cm}=\frac{m_{1}x_{1}+m_{1}x_{2}+\cdot \cdot \cdot }{m_{1}+m_{1}+\cdot
\cdot }  \label{eqnhscoordn}
\end{equation}%
where $x_{n}$ is the position (or distance) of the $nth$ mass from the
origin. Thus for the above example 
\begin{equation}
x_{cm}=\frac{\left( q-p\right) \left( 0\right) +p\left( 1\right) }{\left(
q-p\right) +p}=\frac{p}{q}=\frac{1}{3}
\end{equation}%
So in this case $x_{cm}$ is the value of the fraction $\left( p/q\right) $
which we want to expand in powers of $\left( 1/2\right) $. Henceforth, we
shall denote $x_{cm}$ by $X^{\left( p/q,r/s\right) }$ to make the notation
self explanatory. To carry out the expansion, locate two units of mass, $%
q-p=2$, at the origin and one unit of mass, \thinspace $p=1$, at a unit
distance from the origin.\ We take the zero approximation for the center of
mass, $X_{0}^{\left( 1/3,1/2\right) }$, to be the location of the larger of
the two numbers $p$ and $\left( q-p\right) $, which in this case $x=0$. Now
we consider the system made up of only one unit mass from each side. Clearly
this system of the two units of mass has a center of mass at $1/2$ unit from
the origin (this value is the first order correction to the zero
approximation). To obtain the first approximation $X_{1}^{\left(
1/3,1/2\right) }$ of the center of mass we add this value to $X_{0}^{\left(
1/3,1/2\right) }$. \ Hence, $X_{1}^{\left( 1/3,1/2\right) }=X_{0}^{\left(
1/3,1/2\right) }+1/2=0+1/2$. Next we place the two masses we selected at $%
X_{1}$ and leave the remaining $q-p=1$ unit of mass in its place at the
origin. To get the second approximation we consider one unit of mass from
each side of the new configuration. Clearly the center of mass of the two
unit masses lies in the middle point between $X_{0}^{\left( 1/3,1/2\right)
}=0$ and $X_{1}^{\left( 1/3,1/2\right) }=1/2$, that is at $x=\left(
1/2\right) \left( 1/2\right) $. Thus the second order correction is $-\left(
1/2\right) ^{2}$, the minus sign is needed because the center of mass now
shifts to the left\footnote{%
It is may be advantages to introduce an iterative shift function defined by $%
S\left( n\right) =Q$, where $Q=-1$ if the $n$th iteration shifts the center
of mass to the left and $Q=+1$ if the $n$th iteration shifts the center of
mass to the right. The reader should note that $S\left( n\right) $ keeps
track of the sign inserted at each link in the infinite chain. Here this
function is of no real significance but when the analogy between our system
of $q$ units of mass, our method and a system of $q$ particles and quantum
theory is explored, this function exhibits a property similar to that of
parity operator in quantum mechanics.}. The second approximation now reads $%
X_{2}^{\left( 1/3,1/2\right) }=X_{1}^{\left( 1/3,1/2\right) }-\left(
1/2\right) \left( 1/2\right) =0+\left( 1/2\right) -\left( 1/2\right) \left(
1/2\right) $. Placing these two units of mass at $X_{2}^{\left( 1/3\right) }$
and repeating the procedure $n$ times, we get%
\begin{eqnarray}
X_{n}^{\left( 1/3,1/2\right) } &=&0+\left( \frac{1}{2}\right) -\left( \frac{1%
}{2}\right) ^{2}+\left( \frac{1}{2}\right) ^{3}-\left( \frac{1}{2}\right)
^{4}+\cdot \cdot \cdot +\left( -1\right) ^{n+1}\left( \frac{1}{2}\right) ^{n}
\notag \\
&=&\allowbreak \frac{1}{3}-\frac{1}{3}\left( -\frac{1}{2}\right) ^{n}
\label{example1/3}
\end{eqnarray}%
and in the limit of $n\rightarrow \infty $, we have%
\begin{equation}
X^{\left( 1/3,1/2\right) }=\lim_{n\rightarrow \infty }X_{n}^{\left(
1/3\right) }=\frac{1}{3}  \label{eqnex1/3se1/2}
\end{equation}%
Thus our method yields the desired expansion of the fraction $1/3$ in powers
of $\left( 1/2\right) $%
\begin{equation}
\frac{1}{3}=\sum_{k=1}^{\infty }\left( -1\right) ^{k+1}\left( \frac{1}{2}%
\right) ^{k}  \notag
\end{equation}%
The $n$th partial sum $X_{n}^{\left( 1/3,1/2\right) }$ in (\ref{example1/3})
for the expansion of $1/3$ in powers of $\left( 1/2\right) $ is thus 
\begin{equation}
X_{n}^{\left( 1/3,1/2\right) }=\frac{1}{3}-\frac{1}{3}\left( \left( -\frac{1%
}{2}\right) \right) ^{n}=\frac{1}{2^{n}}\frac{2^{n}-\left( -1\right) ^{n}}{3}%
\text{, \ \ \ \ \ \ \ \ }n\geq 0  \label{eqnsum1/3}
\end{equation}%
At this point the reader may recognize that the sequence defined by $%
2^{n}X_{n}^{\left( 1/3,1/2\right) }$%
\begin{equation}
2^{n}X_{n}^{\left( 1/3,1/2\right) }=\frac{2^{n}-\left( -1\right) ^{n}}{3}=%
\frac{2^{n}-\left( -1\right) ^{n}}{2+1}  \label{eqjacobnumbdef}
\end{equation}%
is the well known Jacobsthal sequence (Sloane \textbf{A001045)} whose
members are the Jacobsthal numbers ($J_{n}$)%
\begin{equation}
0,1,1,3,5,11,21,43,85,171,341,683,.....  \notag
\end{equation}%
Thus our $n$-th order approximation of the center of mass $X_{n}^{\left(
1/3,1/2\right) }$ (or the partial sums in the expansion of $1/3$ in powers
of $1/2$) are related to the Jacobsthal numbers $J_{n}$ through the relation 
\begin{equation}
J_{n}=2^{n}X_{n}^{\left( 1/3,1/2\right) }  \label{eqrebtjandus}
\end{equation}%
This also can be seen by making the substitution $X_{n}^{\left(
1/3,1/2\right) }=J_{n}/2^{n}$ in equations (\ref{eqnsum1/3}) to obtain 
\begin{equation}
J_{n}=\frac{2^{n}-\left( -1\right) ^{n}}{3}=\frac{2^{n}-\left( -1\right) ^{n}%
}{2+1}  \label{A001045}
\end{equation}%
which is the closed form for the Jacobsthal numbers $J_{n}$.

\section{Generalization of Jacobsthal numbers}

The example in the previous section suggest that we may have found a natural
way to generalize the Jacobsthal sequence. To explore this point further,
let's consider the following examples. Let us expand $1/4$ in powers of $%
\left( 1/3\right) $. This can be carried out by the CM method by placing $3$
units of mass at $x=0$ and one unit of mass at $x=1$. Let the zero
approximation for the center of mass $X_{0}^{\left( 1/4,1/3\right) }$ be the
position of the larger collection of unit masses,i.e., $x=0$. To obtain
higher corrections always combine $2$ units of mass from the larger set of
units of mass with the remaining single unit of mass and follow the steps in
the previous example. This leads to the following series for the center of
mass

\begin{equation*}
\frac{1}{4}=0+\sum_{n=1}^{\infty }\left( -1\right) ^{n+1}\left( \frac{1}{3}%
\right) ^{n}
\end{equation*}%
whose $n$th partial sum $X_{n}^{\left( 1/4,1/3\right) }$ is given by%
\begin{equation}
X_{n}^{\left( 1/4,1/3\right) }=\frac{1}{4}-\frac{1}{4}\left( \left( -\frac{1%
}{3}\right) \right) ^{n}=\frac{1}{\left( 3\right) ^{n}}\frac{\left( 3\right)
^{n}-\left( -1\right) ^{n}}{3+1},\text{ \ \ \ \ \ \ \ \ \ \ }n\geq 0  \notag
\end{equation}%
The sequence $S_{n}^{\left( 1/4,1/3\right) }=3^{n}X_{n}^{\left(
1/4,1/3\right) }$%
\begin{equation}
S_{n}=\frac{\left( 3\right) ^{n}-\left( -1\right) ^{n}}{3+1},\text{ \ \ \ \
\ \ \ \ \ \ }n=0,1,2,3,....  \label{eqnGenJacfor1/3+1}
\end{equation}%
satisfying the recursion relation%
\begin{equation*}
S_{n}^{\left( 1/4,1/3\right) }=\left\{ 
\begin{array}{c}
0\text{ \ \ \ \ \ \ \ \ \ \ \ \ \ \ \ \ \ \ \ \ \ \ \ \ \ \ \ \ \ \ \ \ \ \
\ \ }n=0 \\ 
1\text{\ \ \ \ \ \ \ \ \ \ \ \ \ \ \ \ \ \ \ \ \ \ \ \ \ \ \ \ \ \ \ \ \ \ \
\ \ \ }n=0 \\ 
2S_{n-1}^{\left( 1/4,1/3\right) }+3S_{n-2}^{\left( 1/4,1/3\right) }\text{ \
\ \ \ \ \ \ \ }n>1%
\end{array}%
\right.
\end{equation*}%
is Sloane \textbf{A015518, }whose members are%
\begin{equation}
0,1,2,7,20,61,182,547,1640,4921,14762,44287,132860,398581,...
\end{equation}%
Notice that equation (\ref{eqnGenJacfor1/3+1}) has a similar form to the
closed form for the Jacobsthal numbers in (\ref{A001045}).

Likewise if we expand $1/5$ in powers of $\left( 1/4\right) $ using our
method we obtain the infinite series%
\begin{equation}
\frac{1}{5}=0+\sum_{n=1}^{\infty }\left( -1\right) ^{n+1}\left( \frac{1}{4}%
\right) ^{n}  \label{eqex1/5inpo1/4}
\end{equation}%
whose $n$th partial sum is given by%
\begin{eqnarray*}
X_{n}^{\left( 1/5,1/4\right) } &=&\allowbreak \frac{1}{5}-\frac{1}{5}\left(
\left( -\frac{1}{4}\right) \right) ^{k} \\
&=&\frac{1}{\left( 4\right) ^{n}}\frac{\left( 4\right) ^{n}-\left( -1\right)
^{n}}{4+1}
\end{eqnarray*}%
The sequence defined by $S_{n}^{\left( 1/5,1/4\right) }=4^{n}X_{n}^{\left(
1/5,1/4\right) }$ reads%
\begin{equation}
S_{n}^{\left( 1/5,1/4\right) }=\frac{4^{n}-\left( -1\right) ^{n}}{4+1}
\label{eqnGenJacfor1/4+1}
\end{equation}%
or 
\begin{equation}
0,1,3,13,51,205,819,3277,13107,52429,209715,838861,3355443,....
\end{equation}%
which is Sloane \textbf{A015521}, satisfying the following recursion
relations%
\begin{equation*}
S_{n}^{\left( 1/5,1/4\right) }=\left\{ 
\begin{array}{c}
0\text{ \ \ \ \ \ \ \ \ \ \ \ \ \ \ \ \ \ \ \ \ \ \ \ \ \ \ \ \ \ \ \ \ \ \
\ \ }n=0 \\ 
1\text{\ \ \ \ \ \ \ \ \ \ \ \ \ \ \ \ \ \ \ \ \ \ \ \ \ \ \ \ \ \ \ \ \ \ \
\ \ \ }n=0 \\ 
3S_{n-1}^{\left( 1/5,1/4\right) }+4S_{n-2}^{\left( 1/5,1/4\right) }\text{ \
\ \ \ \ \ \ \ }n>1%
\end{array}%
\right.
\end{equation*}%
Notice that equation (\ref{eqnGenJacfor1/4+1}) has a similar form to the
closed form for the Jacobsthal numbers in (\ref{A001045}).

Next we consider the expansion of the fraction $1/5$ in powers of $\left(
2/3\right) $. To do this, choose iterative weight $w=2:1$, and apply our
method%
\begin{equation}
\frac{1}{5}=0+\left( \frac{1}{3}\right) -\left( \frac{2}{3}\right) \left( 
\frac{1}{3}\right) +\left( \frac{2}{3}\right) ^{2}\left( \frac{1}{3}\right)
-\left( \frac{2}{3}\right) ^{3}\left( \frac{1}{3}\right) +...
\label{eqexpa1/5inp2/3}
\end{equation}%
The $n$th partial sum is given by%
\begin{equation}
X_{n}^{\left( 1/5,2/3\right) }=\frac{1}{5}-\frac{1}{5}\left( -\frac{2}{3}%
\right) ^{n}=\frac{1}{3^{n}}\frac{\left( 3\right) ^{n}-\left( -2\right) ^{n}%
}{3+2}\text{, \ \ \ \ }n\geq 0
\end{equation}%
The $n$th partial sum $S_{n}^{\left( 1/5,2/3\right) }$ defined by%
\begin{equation}
S_{n}^{\left( 1/5,2/3\right) }\equiv 3^{n}X_{n}^{\left( 1/5,2/3\right) }=%
\frac{\left( 3\right) ^{n}-\left( -2\right) ^{n}}{3+2}
\end{equation}%
is precisely the closed form for Sloane \textbf{A015441} (\textbf{%
Generalized Fibonacci sequence}) whose members are%
\begin{equation}
0,1,1,7,13,55,133,463,1261,4039,....
\end{equation}%
which satisfy the three recursion relation%
\begin{equation}
S_{n}^{\left( 1/5,2/3\right) }=\left\{ 
\begin{array}{c}
0\text{ \ \ \ \ \ \ \ \ \ \ \ \ \ \ \ \ \ \ \ \ \ \ \ \ \ \ \ \ \ \ \ \ \ \ }%
n=0 \\ 
1\text{ \ \ \ \ \ \ \ \ \ \ \ \ \ \ \ \ \ \ \ \ \ \ \ \ \ \ \ \ \ \ \ \ \ \ }%
n=1 \\ 
S_{n-1}^{\left( 1/5,2/3\right) }+6S_{n-2}^{\left( 1/5,2/3\right) }\text{\ \
\ \ \ \ \ }n>1%
\end{array}%
\right.
\end{equation}%
The last example we consider in this section is the expansion of $1/7$ in
powers of $\left( 3/4\right) $. This can be achieved by choosing the
iterative weight $w=3:1$, and then applying our method. Thus skipping the
algebra, we get the infinite series%
\begin{equation}
\frac{1}{7}=\sum_{n=1}^{\infty }\frac{\left( -1\right) ^{n+1}}{3}\left( 
\frac{3}{4}\right) ^{n}  \notag
\end{equation}%
whose $n$th partial sum read%
\begin{equation}
X_{n}^{\left( 1/7,3/4\right) }=\allowbreak \frac{1}{7}-\frac{1}{7}\left( -%
\frac{3}{4}\right) ^{k}=\frac{1}{\left( 4\right) ^{n}}\frac{\left( 4\right)
^{n}-\left( -3\right) ^{n}}{4+3},\text{ \ \ \ \ \ }n\geq 0  \notag
\end{equation}%
In this case the numbers defined by $S_{n}^{\left( 1/7,3/4\right) }=\left(
4\right) ^{n}X_{n}^{\left( 1/7,3/4\right) }$, 
\begin{equation}
S_{n}^{\left( 1/7,3/4\right) }=\frac{\left( 4\right) ^{n}-\left( -3\right)
^{n}}{4+3},\text{ \ \ \ \ \ }n\geq 0  \notag
\end{equation}%
are the members of \ Sloane \textbf{A053404}. The sequence $S_{n}^{\left(
1/7,3/4\right) }$\ satisfies the recursion relation%
\begin{equation}
S_{n}^{\left( 1/7,3/4\right) }=\left\{ 
\begin{array}{c}
0\text{ \ \ \ \ \ \ \ \ \ \ \ \ \ \ \ \ \ \ \ \ \ \ \ \ \ \ \ \ \ \ \ \ \ \
\ }n=0 \\ 
1\text{ \ \ \ \ \ \ \ \ \ \ \ \ \ \ \ \ \ \ \ \ \ \ \ \ \ \ \ \ \ \ \ \ \ \
\ \ }n=1 \\ 
S_{n-1}^{\left( 1/7,3/4\right) }+12S_{n-2}^{\left( 1/7,3/4\right) }\text{\ \
\ \ \ \ \ \ }n>1%
\end{array}%
\right.
\end{equation}

In general for any integer $q\left( s,r\right) =s+r$, our method provides
the expansion of the rational number $1/q$ in powers of $\left( r/s\right) $
with $s>r>0$ and $r,s$ positive integers\footnote{%
Later we will relax the condition $q=r+s$.}. Our method leads to an $n$th
partial sums of the form $X_{n}^{\left( \frac{1}{q},\frac{r}{s}\right) }$ 
\begin{equation}
X_{n}^{\left( \frac{1}{q},\frac{r}{s}\right) }=\frac{1-\left( -r/s\right)
^{n}}{s+r}=\frac{1}{s^{n}}\frac{\left( s\right) ^{n}-\left( -r\right) ^{n}}{%
s+r},\text{ \ \ \ \ }q=r+s  \label{eqmoregeAbdnumbr/r+s}
\end{equation}%
for $n\geq 0$.

The sequence defined by $S_{n}^{\left( \frac{1}{q},\frac{r}{s}\right)
}=s^{n}X_{n}^{\left( \frac{1}{q},\frac{r}{s}\right) }$ 
\begin{equation}
S_{n}^{\left( \frac{1}{q},\frac{r}{s}\right) }=\frac{\left( s\right)
^{n}-\left( -r\right) ^{n}}{s+r},\text{ \ \ \ \ }q=r+s
\label{eqnGenJacfor1/q+1}
\end{equation}%
has the same form as the Jacobsthal sequence. This form offers a natural
generalization of the Jacobsthall numbers in (\ref{A001045}). Thus we
propose the following definition for "the $n$th more-generalized $q-$%
Jacobsthall number."

\begin{definition}
For any integers \ $s>r>0$, $q=s+r$ and any nonnegative integer $n$, we
define the $nth$\ more-generalized $q-$Jacobsthal number by%
\begin{equation}
J_{n}^{\left( \frac{1}{q},\frac{r}{s}\right) }\equiv s^{n}X_{n}^{\left( 
\frac{1}{q},\frac{r}{s}\right) }=\frac{\left( s\right) ^{n}-\left( -r\right)
^{n}}{s+r},\text{ \ \ \ \ }q=r+s  \label{eqmoregejacnumbr/r+s}
\end{equation}
\end{definition}

\begin{lemma}
The $n$th more-generalized $q-$Jacobsthal numbers satisfy the following
recursion relations%
\begin{equation}
J_{n+1}^{\left( \frac{1}{q},\frac{r}{s}\right) }=sJ_{n}^{\left( \frac{1}{q},%
\frac{r}{s}\right) }+\left( -r\right) ^{n},\text{\ \ \ \ \ \ \ \ }n\geqslant
0
\end{equation}%
\begin{equation}
J_{n+1}^{\left( \frac{1}{q},\frac{r}{s}\right) }=s^{n}-rJ_{n}^{\left( \frac{1%
}{q},\frac{r}{s}\right) },\text{\ \ \ \ \ \ \ \ \ \ \ }n\geqslant 0
\end{equation}%
and%
\begin{equation}
J_{n}^{\left( \frac{1}{q},\frac{r}{s}\right) }=\left\{ 
\begin{array}{c}
0,\text{ \ \ \ \ \ \ \ \ \ \ \ \ \ \ \ \ \ \ \ \ }n=0 \\ 
1,\text{\ \ \ \ \ \ \ \ \ \ \ \ \ \ \ \ \ \ \ \ \ }n=1 \\ 
\left( s-r\right) J_{n-1}^{\left( \frac{1}{q},\frac{r}{s}\right) }+\left(
rs\right) J_{n-2}^{\left( \frac{1}{q},\frac{r}{s}\right) },\text{ \ \ }n>1%
\end{array}%
\right.  \label{engenJacsnumbq=r+s-(r/s)}
\end{equation}
\end{lemma}

\begin{theorem}
The generating function for the $nth$\ more generalized $q-$Jacobsthal
numbers is given by%
\begin{equation}
J^{\left( \frac{1}{q},\frac{r}{s}\right) }\left( z\right)
=\sum\limits_{n=0}^{\infty }J_{n}^{\left( \frac{1}{q},\frac{r}{s}\right)
}z^{n}=\frac{z}{1-\left( s-r\right) z-rsz^{2}}
\end{equation}
\end{theorem}

\begin{proof}
The recurrence formula in (\ref{engenJacsnumbq=r+s-(r/s)}) is a special case
of \textbf{Lucas }sequences of the first kind $U_{n}(P,Q)$ defined by the
recursion relation%
\begin{eqnarray}
U_{0}(P,Q) &=&0  \notag \\
U_{1}(P,Q) &=&1  \notag \\
U_{n}(P,Q) &=&PU_{n-1}(P,Q)-QU_{n-2}(P,Q)  \label{eqgwnrejatooursluks}
\end{eqnarray}%
for $P=\left( s-r\right) $ and $Q=-rs$. \ The ordinary generating function
for $U_{n}(P,Q)$ in (\ref{eqgwnrejatooursluks}) is given by 
\begin{equation}
\sum\limits_{n=0}^{\infty }U_{n}(P,Q)z^{n}=\frac{z}{1-Pz+Qz^{2}}
\label{ealucasseqgefu}
\end{equation}%
Substituting $P=\left( s-r\right) $ and $Q=-rs$ in (\ref{ealucasseqgefu}),
we obtain the desired result for $J_{n}^{\left( \frac{1}{q},\frac{r}{s}%
\right) }$
\end{proof}

At this point it is worth observing that for the special case in the
expansion of $1/3$ in powers of $\left( 1/2\right) $, we have $q=3,r=1,s=2$
and our definition (\ref{eqmoregejacnumbr/r+s}) reduces to 
\begin{equation}
J_{n}^{\left( \frac{1}{3},\frac{1}{2}\right) }=\frac{\left( 2\right)
^{n}-\left( -1\right) ^{n}}{2+1}  \label{eqdeofnthJacNumber}
\end{equation}%
which is the formula for Jacobsthall numbers $J_{n}$

For the case of $q=s+r=s+1$, the $nth$\textit{\ }more generalized $q$%
-Jacobsthal numbers in (\ref{eqmoregejacnumbr/r+s}) reduce to%
\begin{equation}
J_{n}^{\left( \frac{1}{s+1},\frac{1}{s}\right) }=\frac{\left( s\right)
^{n}-\left( -1\right) ^{n}}{s+1}=J_{sn}
\end{equation}%
the generalized Jacobsthal number, $J_{sn}$, introduced in \cite{Ji Young
Choi}. Thus we see that the generalized Jacobsthal numbers introduced in 
\cite{Ji Young Choi} are still a special case of the $n$th more-generalized $%
q-$Jacobsthal numbers defined by equation (\ref{eqmoregejacnumbr/r+s}).

\section{\protect\bigskip Continuation to negative value of $n$}

We can extend the definition of the more-generalized $q-$Jacobsthal numbers
in (\ref{eqmoregejacnumbr/r+s}) to negative values of $n$. This we can do by
assuming that the recursion relation in (\ref{engenJacsnumbq=r+s-(r/s)})
provides a continuation to negative values of $n$.

\begin{theorem}
For any nonnegative integer $n$, we have 
\begin{equation}
J_{-n}^{\left( \frac{1}{q},\frac{r}{s}\right) }=-\left( \frac{-1}{rs}\right)
^{n}\frac{s^{n}-\left( -r\right) ^{n}}{s+r}  \label{Negoregejacnumbr/r+s}
\end{equation}
\end{theorem}

\begin{proof}
Solve (\ref{engenJacsnumbq=r+s-(r/s)}) for $J_{n-2}^{\left( \frac{1}{q},%
\frac{r}{s}\right) }$ 
\begin{equation}
J_{n-2}^{\left( \frac{1}{q},\frac{r}{s}\right) }=-\frac{s-r}{rs}%
J_{n-1}^{\left( \frac{1}{q},\frac{r}{s}\right) }+\frac{1}{rs}J_{n}^{\left( 
\frac{1}{q},\frac{r}{s}\right) }  \label{MoreGenjactonen}
\end{equation}%
For $n=1$, the above equation gives 
\begin{eqnarray}
J_{-1}^{\left( \frac{1}{q},\frac{r}{s}\right) } &=&-\frac{s-r}{rs}%
J_{0}^{\left( \frac{1}{q},\frac{r}{s}\right) }+\frac{1}{rs}J_{1}^{\left( 
\frac{1}{q},\frac{r}{s}\right) }  \notag \\
&=&0+\frac{1}{rs}
\end{eqnarray}%
The two initial conditions 
\begin{eqnarray}
J_{0}^{\left( \frac{1}{q},\frac{r}{s}\right) } &=&0\ \text{\ } \\
J_{-1}^{\left( \frac{1}{q},\frac{r}{s}\right) } &=&\frac{1}{rs}
\end{eqnarray}%
and the recursion relation in (\ref{MoreGenjactonen}) now fix the value of $%
J_{-n}^{\left( \frac{1}{q},\frac{r}{s}\right) }$ for all $n>1$. Thus with
little effort we obtain the result in (\ref{Negoregejacnumbr/r+s})
\end{proof}

\begin{lemma}
For any nonnegative integer $n$, 
\begin{equation*}
J_{-n}^{\left( \frac{1}{q},\frac{r}{s}\right) }=-\left( \frac{-1}{rs}\right)
^{n}J_{n}^{\left( \frac{1}{q},\frac{r}{s}\right) }
\end{equation*}
\end{lemma}

This result follows at once from the above theorem and the definition of $%
J_{n}^{\left( \frac{1}{q},\frac{r}{s}\right) }$ in (\ref%
{eqmoregejacnumbr/r+s}). The numbers $J_{-n}^{\left( \frac{1}{q},\frac{r}{s}%
\right) }$ are related to the partial sums $X_{n}^{\left( \frac{1}{q},\frac{r%
}{s}\right) }$ obtained in (\ref{eqmoregeAbdnumbr/r+s})

\begin{equation}
J_{-n}^{\left( \frac{1}{q},\frac{r}{s}\right) }=-\left( -\frac{1}{r}\right)
^{n}X_{n}^{\left( \frac{1}{q},\frac{r}{s}\right) },\text{ \ \ \ \ }%
n\geqslant 0
\end{equation}

\begin{theorem}
For any nonnegative integer $n$, the $nth$\textit{\ }more-generalized $q$%
-Jacobsthal numbers satisfy the recurrence relations 
\begin{eqnarray}
J_{-n}^{\left( \frac{1}{q},\frac{r}{s}\right) } &=&\frac{1}{s^{n+1}}-\text{\ 
}rJ_{-\left( n+1\right) }^{\left( \frac{1}{q},\frac{r}{s}\right) }\text{\ \
\ \ \ \ \ \ \ \ \ \ \ \ \ \ \ \ }n\geqslant 0 \\
J_{-n}^{\left( \frac{1}{q},\frac{r}{s}\right) } &=&\text{\ }\frac{\left(
-1\right) ^{n+1}}{r^{n+1}}+sJ_{-\left( n+1\right) }^{\left( \frac{1}{q},%
\frac{r}{s}\right) }\text{\ \ \ \ \ \ \ \ \ \ \ \ }n\geqslant 0\text{ \ }
\end{eqnarray}%
and%
\begin{equation}
J_{-n}^{\left( \frac{1}{q},\frac{r}{s}\right) }=\left\{ 
\begin{array}{c}
0\text{ \ \ \ \ \ \ \ \ \ \ \ \ \ \ \ \ \ \ \ \ \ \ \ \ \ \ \ \ \ \ \ \ \ }%
n=0 \\ 
\frac{1}{rs}\text{\ \ \ \ \ \ \ \ \ \ \ \ \ \ \ \ \ \ \ \ \ \ \ \ \ \ \ \ \
\ \ \ }n=1 \\ 
-\frac{s-r}{rs}J_{-\left( n-1\right) }^{\left( \frac{1}{q},\frac{r}{s}%
\right) }+\frac{1}{rs}J_{-\left( n-2\right) }^{\left( \frac{1}{q},\frac{r}{s}%
\right) }\text{ \ \ \ \ \ \ }n>1%
\end{array}%
\right.
\end{equation}
\end{theorem}

The recursion relations for $J_{-n}^{\left( \frac{1}{q},\frac{r}{s}\right) }$
may be obtained by solving for $J_{n}^{\left( \frac{1}{q},\frac{r}{s}\right)
}$ in terms of $J_{-n}^{\left( \frac{1}{q},\frac{r}{s}\right) }$ and then
substituting in the recursion relations for $J_{n}^{\left( \frac{1}{q},\frac{%
r}{s}\right) }$.

Using the formulas in equations (\ref{eqmoregejacnumbr/r+s}) and (\ref%
{Negoregejacnumbr/r+s}), we obtain the complete set of the Jacobsthal
numbers arising in the expansion of $1/q=1/\left( s+r\right) =1/\left(
2+1\right) $ in powers of $\left( r/s\right) =\left( 1/2\right) $ 
\begin{equation}
...,\frac{43}{128},-\frac{21}{64},\frac{11}{32},-\frac{5}{16},\frac{3}{8},-%
\frac{1}{4},\frac{1}{2},0,1,1,3,5,11,21,43,.....
\end{equation}%
It is worth noting that our extension of the Jacobsthal numbers and infinity
many other sequences of numbers show that their use gives a completeness and
insight that otherwise is lacking in the simpler theory of Jacobsthal
sequences for positive $n$.

\section{\textbf{Jacobsthal-like numbers}.}

\bigskip To introduce what we mean by Jacobsthal-like numbers, let's
consider the following example of expanding $1/7$ in powers $\left(
1/2\right) $ and $\left( 1/8\right) $. The expansion in powers of $\left(
1/2\right) $ can be carried out by the center of mass method by placing $6$
units of mass at $x=0$ and one unit of mass at $x=1$ and choosing the zero
approximation for the center of mass $X_{0}$ to be the position of the
larger collection ($x=0$). Setting the iterative weight $w=1:1$, our method
yields the infinite series%
\begin{eqnarray}
\frac{1}{7} &=&X^{\left( 1/7,1/2\right) }=0  \notag \\
&&+\left( \frac{1}{2}\right) -\left( \frac{1}{2}\right) ^{2}-\left( \frac{1}{%
2}\right) ^{3}+\left( \frac{1}{2}\right) ^{4}-\left( \frac{1}{2}\right)
^{5}-\left( \frac{1}{2}\right) ^{6}  \notag \\
&&+\left( \frac{1}{2}\right) ^{7}-\left( \frac{1}{2}\right) ^{8}-\left( 
\frac{1}{2}\right) ^{9}+\left( \frac{1}{2}\right) ^{10}-\left( \frac{1}{2}%
\right) ^{11}-\left( \frac{1}{2}\right) ^{12}  \notag \\
&&+......  \label{eqexpan1/7-1/2}
\end{eqnarray}%
whose partial sums $X_{n}^{\left( 1/7,1/2\right) }$ are given by%
\begin{eqnarray}
&&0,\frac{1}{2},\frac{1}{2^{2}},\frac{1}{2^{3}},\frac{3}{2^{4}},\frac{5}{%
2^{5}},\frac{9}{2^{6}},\frac{19}{2^{7}},\frac{37}{2^{8}},\frac{73}{2^{9}},%
\frac{147}{2^{10}},\frac{293}{2^{11}},\frac{585}{2^{12}},\frac{1171}{2^{13}},%
\frac{2341}{2^{14}},  \notag \\
&&\frac{4681}{2^{15}},\frac{9363}{2^{16}},\frac{18725}{2^{17}},\frac{37449}{%
2^{18}},\frac{74899}{2^{19}},\frac{149797}{2^{21}},\frac{299593}{2^{22}},....
\label{eqnpartsumsof1/7-1/2}
\end{eqnarray}%
The subsequence $S_{n}=$ $2^{3n}X_{3n}^{\left( 1/7,1/2\right) }$ of the
above sequence 
\begin{equation}
0,1,9,73,585,4681,37449,299593,2396745,19173961,153391689,...
\end{equation}%
is Sloane \textbf{A023001}%
\begin{equation}
\frac{\left( 8\right) ^{n}-1}{8-1},\text{ }n=0,1,2,3,...
\end{equation}%
This sequence may be obtained directly by rearranging the terms in the sum
in equation (\ref{eqexpan1/7-1/2}) in the following way

\begin{eqnarray*}
\frac{1}{7} &=&X^{\left( 1/7,1/2\right) }=\left[ 0+\left( \frac{1}{2}\right)
-\left( \frac{1}{2}\right) ^{2}-\left( \frac{1}{2}\right) ^{3}\right] \\
&&\times \left[ 1+\left( \frac{1}{2}\right) ^{3}+\left( \frac{1}{2}\right)
^{6}+\left( \frac{1}{2}\right) ^{9}+\left( \frac{1}{2}\right) ^{12}+....%
\right] \\
&=&\sum_{k=0}^{\infty }\left( \frac{1}{2}\right)
^{3k+3}=\sum\limits_{n=1}^{\infty }\left( \frac{1}{8}\right) ^{n}
\end{eqnarray*}%
so we obtain an expansion of $1/7$ in powers of $\left( 1/8\right) $. Since
the right hand side is given in powers of $1/8$, we denote the $n$th partial
sum in the right hand side by $\widetilde{X}_{n}^{\left( 1/7,1/8\right) }$.
We note here that the $n$th partial sum\footnote{%
Notice that I extended the rang of $n>0$ to $n\geq 0$ where I have assigned
value of zero to $0$th partial sum. This I have done to keep the notation
consistent with that in the literature.}

\begin{equation}
\widetilde{X}_{n}^{\left( 1/7,1/8\right) }=\frac{1}{7}-\frac{1}{7}\left( 
\frac{1}{8}\right) ^{n}=\frac{1}{\left( 8\right) ^{n}}\frac{\left( 8\right)
^{n}-1}{8-1},\text{ \ \ \ \ \ \ \ \ \ \ }n\geq 0
\end{equation}%
is equal to the partial sums $X_{3n}^{\left( 1/7,1/2\right) }$ obtained in (%
\ref{eqnpartsumsof1/7-1/2}). Observe here we have $r=1,s=8$ and so $\left(
s-r\right) =7$, so in this case the partial sums have the form%
\begin{equation}
\widetilde{X}_{n}^{\left( 1/7,1/8\right) }=\frac{1}{\left( s\right) ^{n}}%
\frac{\left( s\right) ^{n}-\left( r\right) ^{n}}{s-r},\text{ \ \ \ \ \ \ \ \
\ \ }n\geq 0
\end{equation}%
or if we let $\widetilde{S}_{n}^{\left( 1/7,1/8\right) }=8^{n}\widetilde{X}%
_{n}^{\left( 1/7,1/8\right) }$, we get 
\begin{equation}
\widetilde{S}_{n}^{\left( 1/7,1/8\right) }=\frac{\left( 8\right) ^{n}-1}{7}=%
\frac{\left( 8\right) ^{n}-1}{8-1},\text{ \ \ \ \ \ \ \ \ \ \ }n\geq 0
\end{equation}%
which is Sloane \textbf{A023001} whose members are%
\begin{equation}
0,1,9,73,585,4681,37449,299593,2396745,19173961,153391689,...
\end{equation}

Next we consider the product $\widetilde{S}_{n}^{\left( \frac{1}{\widetilde{q%
}},\frac{r}{s}\right) }$ defined by the product $s^{n}\cdot \widetilde{X}%
_{n}^{\left( \frac{1}{\widetilde{q}},\frac{r}{s}\right) }$ 
\begin{equation*}
\widetilde{S}_{n}^{\left( \frac{1}{\widetilde{q}},\frac{r}{s}\right) }=\frac{%
s^{n}-r^{n}}{s-r}
\end{equation*}%
where $\widetilde{q}=s-r$. We now define the \textbf{more-generalized
Jacobsthal-like numbers}.\textit{\ }

\begin{definition}
For any nonnegative integer $n$, the $n$th more-generalized $q-$
Jacobsthall-like numbers, $\widetilde{J}_{n}^{\left( \frac{1}{q},\frac{r}{s}%
\right) }$, are defined by
\end{definition}

\begin{equation}
\widetilde{J}_{n}^{\left( \frac{1}{\widetilde{q}},\frac{r}{s}\right) }=s^{n}%
\widetilde{X}_{n}^{\left( \frac{1}{\widetilde{q}},\frac{r}{s}\right) }=\frac{%
s^{n}-r^{n}}{s-r}  \label{eqdeofnthJacNumberDual}
\end{equation}%
Observe if we let $r\rightarrow -r$ in equation (\ref{eqdeofnthJacNumberDual}%
), we recover the $n-$th more-generalized $q-$Jacobsthall numbers, $%
J_{n}^{\left( \frac{1}{q},\frac{r}{s}\right) }$ defined in (\ref%
{eqmoregejacnumbr/r+s}). Thus we have the following corollary:

\begin{corollary}
For any nonnegative integers $n,s,r$ and $s>r$,
\end{corollary}

\begin{equation}
J_{n}^{\left( \frac{1}{q\left( s,r\right) },\frac{r}{s}\right) }=\widetilde{J%
}_{n}^{\left( \frac{1}{\widetilde{q}\left( s,-r\right) },\frac{\left(
-r\right) }{s}\right) }
\end{equation}

For $s=3,r=1$, equation (\ref{eqdeofnthJacNumberDual}) gives

\begin{equation}
0,1,4,13,40,121,364,1093,3280,...
\end{equation}%
for the Jacobsthal-like numbers. This is Sloane \textbf{A003462.}

\begin{lemma}
The more-generalized $q-$Jacobsthal-like numbers, $\widetilde{J}_{n}^{\left( 
\frac{1}{q},\frac{r}{s}\right) }$, satisfy the recursion relations%
\begin{eqnarray}
\widetilde{J}_{n+1}^{\left( \frac{1}{\widetilde{q}},\frac{r}{s}\right) } &=&s%
\widetilde{J}_{n}^{\left( \frac{1}{\widetilde{q}},\frac{r}{s}\right) }+r^{n}%
\text{\ \ \ \ \ \ \ \ \ \ \ \ \ \ \ \ \ \ }n\geqslant 0  \notag \\
\widetilde{J}_{n+1}^{\left( \frac{1}{\widetilde{q}},\frac{r}{s}\right) }
&=&s^{n}+r\widetilde{J}_{n}^{\left( \frac{1}{\widetilde{q}},\frac{r}{s}%
\right) }\text{\ \ \ \ \ \ \ \ \ \ \ \ \ \ \ \ \ \ }n\geqslant 0
\label{eqdualjn2rec}
\end{eqnarray}%
and
\end{lemma}

\begin{equation}
\widetilde{J}_{n}^{\left( \frac{1}{\widetilde{q}},\frac{r}{s}\right)
}=\left\{ 
\begin{array}{c}
0\text{ \ \ \ \ \ \ \ \ \ \ \ \ \ \ \ \ \ \ \ \ \ \ \ \ \ \ \ \ \ \ \ }n=0
\\ 
1\text{\ \ \ \ \ \ \ \ \ \ \ \ \ \ \ \ \ \ \ \ \ \ \ \ \ \ \ \ \ \ \ \ }n=1
\\ 
\left( s+r\right) \widetilde{J}_{n-1}^{\left( \frac{1}{\widetilde{q}},\frac{r%
}{s}\right) }-\left( rs\right) \widetilde{J}_{n-2}^{\left( \widetilde{q},%
\frac{r}{s}\right) }\text{ \ \ \ \ \ \ \ }n>1%
\end{array}%
\right.  \label{eqdualjn3rec}
\end{equation}%
The recurrence formula in (\ref{eqdualjn3rec}) is a special case of \textbf{%
Lucas-}sequences of the first kind $U_{n}(P,Q)$ defined by the recursion
relation

\begin{eqnarray}
U_{0}(P,Q) &=&0  \notag \\
U_{1}(P,Q) &=&1  \notag \\
U_{n}(P,Q) &=&PU_{n-1}(P,Q)-QU_{n-2}(P,Q)  \label{eqgwnrejatoours-1}
\end{eqnarray}%
whose generating function is given by 
\begin{equation}
\sum\limits_{n=0}^{\infty }U_{n}(P,Q)z^{n}=\frac{z}{1-Pz+Qz^{2}}
\label{eqgwnrejatoours-1D}
\end{equation}

\begin{theorem}
The generating function for the more-generalized $q-$Jacobsthal-like
numbers, $\widetilde{J}_{n}^{\left( \frac{1}{\widetilde{q}},\frac{r}{s}%
\right) }$, is given by%
\begin{equation}
\widetilde{J}^{\left( \frac{1}{\widetilde{q}},\frac{r}{s}\right) }\left(
z\right) =\sum\limits_{n=0}^{\infty }\widetilde{J}_{n}^{\left( \frac{1}{%
\widetilde{q}},\frac{r}{s}\right) }z^{n}=\frac{z}{1-\left( s+r\right)
z+srz^{2}}
\end{equation}
\end{theorem}

\begin{proof}
Since (\ref{eqdualjn3rec}) is a special case of (\ref{eqgwnrejatoours-1})
with $P=s+r$ and $Q=rs$, the truth of the above theorem follows at once by
making the substitution $P=s+r$ and $Q=rs$ in (\ref{eqgwnrejatoours-1D})
\end{proof}

Let us start with the recurrence in (\ref{eqdualjn3rec}) as our definition
of the more-generalized $q-$Jacobsthal-like numbers, $\widetilde{J}%
_{n}^{\left( \frac{1}{\widetilde{q}},\frac{r}{s}\right) }$. In this case the
characteristic equation associated with the recurrence in (\ref{eqdualjn3rec}%
) is%
\begin{equation}
t^{2}=\left( s+r\right) t-\left( rs\right) \text{ \ }
\end{equation}%
which has the following roots%
\begin{equation}
t_{1}=s\text{, \ \ }t_{2}=r\text{ \ }
\end{equation}%
Now $\widetilde{J}_{n}^{\left( \frac{1}{\widetilde{q}},\frac{r}{s}\right) }$
may be expressed as a linear combination of $t_{1}^{n}$ and $t_{2}^{n}$%
\begin{equation}
\widetilde{J}_{n}^{\left( \frac{1}{q},\frac{r}{s}\right)
}=c_{1}t_{1}^{n}+c_{2}t_{2}^{n}\text{ \ }
\end{equation}%
Using the initial values $\widetilde{J}_{0}^{\left( \frac{1}{\widetilde{q}},%
\frac{r}{s}\right) }=0$ and $\widetilde{J}_{1}^{\left( \frac{1}{\widetilde{q}%
},\frac{r}{s}\right) }=1$, we obtain the values of $c_{1}$ and $c_{2}$,%
\begin{equation}
c_{1}=c_{2}=\frac{1}{s-r}\text{ \ }
\end{equation}%
Thus we obtain 
\begin{equation}
\widetilde{J}_{n}^{\left( \frac{1}{q},\frac{r}{s}\right) }=\frac{s^{n}-r^{n}%
}{s-r}  \label{eqdeofnthJacNumberDual-1}
\end{equation}%
which is the same expression in (\ref{eqdeofnthJacNumberDual}). If we relax
the assumption that $r,s$ are positive integers, then this expression can
produce many of the well known sequences. For example if we let $2s=1+\sqrt{5%
}$ and $2r=1-\sqrt{5}$, then $\widetilde{q}=s-r=\sqrt{5}$, and equation (\ref%
{eqdeofnthJacNumberDual-1}) becomes%
\begin{equation}
\widetilde{J}_{n}^{\left( \frac{1}{\sqrt{5}},\frac{\left( 1-\sqrt{5}\right)
/2}{\left( 1+\sqrt{5}\right) /2}\right) }=\frac{1}{\sqrt{5}}\left[ \left( 
\frac{1+\sqrt{5}}{2}\right) ^{n}-\left( \frac{1-\sqrt{5}}{2}\right) ^{n}%
\right]
\end{equation}%
Computing the first few terms we get 
\begin{equation}
0,1,1,2,3,5,8,13,21,34,55,89,144,...
\end{equation}%
which is the \textbf{Fibonacci sequence}. Thus the Fibonacci sequence $F_{n}$%
\begin{equation}
F_{n}=\widetilde{J}_{n}^{\left( \frac{1}{\sqrt{5}},\frac{\left( 1-\sqrt{5}%
\right) /2}{\left( 1+\sqrt{5}\right) /2}\right) }
\end{equation}%
It is clear \ that the Fibonacci sequence is just a special case of our the
more-generalized $q-$Jacobsthal-like numbers, $\widetilde{J}_{n}^{\left( 
\frac{1}{\widetilde{q}},\frac{r}{s}\right) }$. Moreover substituting $2s=1+%
\sqrt{5}$ and $2r=1-\sqrt{5}$ in (\ref{eqdualjn3rec}), we obtain%
\begin{equation}
F_{n}=\left\{ 
\begin{array}{c}
0\text{ \ \ \ \ \ \ \ \ \ \ \ \ \ \ \ \ \ \ \ \ \ \ \ \ \ \ \ \ \ \ \ }n=0
\\ 
1\text{\ \ \ \ \ \ \ \ \ \ \ \ \ \ \ \ \ \ \ \ \ \ \ \ \ \ \ \ \ \ \ \ }n=1
\\ 
F_{n-1}+F_{n-2}\text{ \ \ \ \ \ \ \ }n>1%
\end{array}%
\right.
\end{equation}%
which is the recurrence equation for the Fibonacci numbers.

If we set $s=1+\sqrt{2}$ and $r=1-\sqrt{2}$, then $\widetilde{q}=s-r=2\sqrt{2%
}$, and the above expression becomes

\begin{equation}
\widetilde{J}_{n}^{\left( \frac{1}{2\sqrt{2}},\frac{\left( 1-\sqrt{2}\right) 
}{\left( 1+\sqrt{2}\right) }\right) }=\frac{1}{2\sqrt{2}}\left[ \left( 1+%
\sqrt{2}\right) ^{n}-\left( 1-\sqrt{2}\right) ^{n}\right]
\end{equation}%
Computing the first few terms we get

\begin{equation}
0,1,2,5,12,29,70,169,408,985,2378,5741,13860,33461,80782,...
\end{equation}%
which is Sloane \textbf{A000129}, also known as \textbf{Pell sequence}. To
be more certain let's derive the recurrence relation for Pell sequence. The
closed form for Pell numbers, $P_{n}$, reads\textbf{\ }%
\begin{equation}
P_{n}=\widetilde{J}_{n}^{\left( \frac{1}{2\sqrt{2}},\frac{\left( 1-\sqrt{2}%
\right) }{\left( 1+\sqrt{2}\right) }\right) }
\end{equation}%
\textbf{\ }If we set $s=1+\sqrt{2}$ and $r=1-\sqrt{2}$ in the recurrence
equation for $\widetilde{J}_{n}^{\left( \frac{1}{\widetilde{q}},\frac{r}{s}%
\right) }$ in (\ref{eqdualjn3rec}) and use the fact that $\widetilde{J}%
_{n}^{\left( \frac{1}{2\sqrt{2}},\frac{\left( 1-\sqrt{2}\right) }{\left( 1+%
\sqrt{2}\right) }\right) }=$ $P_{n}$, we obtain 
\begin{equation}
P_{n}=\left\{ 
\begin{array}{c}
\text{ \ \ \ \ }0\text{ \ \ \ \ \ \ \ \ \ \ \ \ \ \ \ \ \ \ \ \ \ \ \ \ \ }%
n=0 \\ 
\text{ \ \ \ \ }1\text{\ \ \ \ \ \ \ \ \ \ \ \ \ \ \ \ \ \ \ \ \ \ \ \ \ \ \ 
}n=1 \\ 
2\text{\ }P_{n-1}+P_{n-2}\text{ \ \ \ \ \ \ \ }n>1%
\end{array}%
\right.
\end{equation}%
which is precisely the recurrence equation for the Pell numbers. At this
point it is worth making the following remark. Although the recurrence
relations for the more-generalized $q-$Jacobsthal numbers and their $q-$%
Jacobsthal-like were derived for $s,r$ positive integers; they provide the
means for analytic continuations to non integral values of $s$ and $r$. We
will come to this point in some detail at the end of the paper.

\section{Continuation to negative value of $n$}

We can extend the definition of the more-generalized $q-$Jacobsthal-like
numbers to negative values of $n$. This we can do by assuming that the above
recursion relation in (\ref{eqdualjn3rec}) provides a continuation to
negative values of $n$. Thus for $n=1$, 
\begin{equation}
\left( rs\right) \widetilde{J}_{n-2}^{\left( \frac{1}{\widetilde{q}},\frac{r%
}{s}\right) }=\left( s+r\right) \widetilde{J}_{n-1}^{\left( \frac{1}{%
\widetilde{q}},\frac{r}{s}\right) }-\widetilde{J}_{n}^{\left( \frac{1}{%
\widetilde{q}},\frac{r}{s}\right) }\text{ \ \ \ \ }\ \text{\ }
\label{eqnancinofjactonen}
\end{equation}%
gives%
\begin{equation*}
\widetilde{J}_{-1}^{\left( \frac{1}{\widetilde{q}},\frac{r}{s}\right) }\ =-%
\frac{1}{rs}
\end{equation*}%
Now the values $J_{0}^{\widetilde{q}}$ and $J_{-1}^{\widetilde{q}}$ together
with equation (\ref{eqnancinofjactonen}) give 
\begin{equation}
\widetilde{J}_{-n}^{\left( \frac{1}{\widetilde{q}},\frac{r}{s}\right) }=-%
\frac{1}{\left( rs\right) ^{n}}\frac{s^{n}-r^{n}}{s-r}=-\frac{1}{\left(
rs\right) ^{n}}\widetilde{J}_{n}^{\left( \frac{1}{\widetilde{q}},\frac{r}{s}%
\right) }
\end{equation}%
for all $n\geqslant 0$. Thus we have established the following theorem.

\begin{theorem}
\bigskip For any nonnegative integer $n$, we have%
\begin{equation}
\widetilde{J}_{-n}^{\left( \frac{1}{\widetilde{q}},\frac{r}{s}\right) }=-%
\frac{1}{\left( rs\right) ^{n}}\frac{s^{n}-r^{n}}{s-r}
\label{eqdeofnthJacNumberNegativeDual}
\end{equation}
\end{theorem}

\bigskip From equation (\ref{eqdeofnthJacNumberDual}) and (\ref%
{eqdeofnthJacNumberNegativeDual}) that 
\begin{equation}
\widetilde{J}_{-n}^{\left( \frac{1}{\widetilde{q}},\frac{r}{s}\right) }=-%
\frac{1}{\left( rs\right) ^{n}}\widetilde{J}_{n}^{\left( \frac{1}{\widetilde{%
q}},\frac{r}{s}\right) },\text{\ \ \ \ }n\geqslant 0  \label{eqmdualtnegmdu}
\end{equation}%
Using the fact $\widetilde{J}_{n}^{\left( \frac{1}{\widetilde{q}},\frac{r}{s}%
\right) }=s^{n}\widetilde{X}_{n}^{\left( \frac{1}{\widetilde{q}},\frac{r}{s}%
\right) }$ and (\ref{eqmdualtnegmdu}), we obtain

\begin{equation}
\widetilde{J}_{-n}^{\left( \frac{1}{\widetilde{q}},\frac{r}{s}\right)
}=-\left( \frac{1}{r}\right) ^{n}\widetilde{X}_{n}^{\left( \frac{1}{%
\widetilde{q}},\frac{r}{s}\right) },\text{ \ \ \ \ }n\geqslant 0
\end{equation}%
relating the $n$th more-generalized $q-$Jacobsthal-like numbers for negative 
$n$ to the $n$th partial sum in the expansion of $\left( 1/\left( s-r\right)
\right) $ in powers of $\left( r/s\right) $. We conclude this section with
following result.

\begin{lemma}
For any nonnegative integer $n$, the $n$th more-generalized $q-$%
Jacobsthal-like numbers satisfy the following recurrence relations
\end{lemma}

\begin{eqnarray}
\widetilde{J}_{-\left( n+1\right) }^{\left( \frac{1}{\widetilde{q}},\frac{r}{%
s}\right) } &=&\frac{1}{r}\widetilde{J}_{n}^{\left( \frac{1}{\widetilde{q}},%
\frac{r}{s}\right) }-\frac{1}{rs}\frac{1}{s^{n}}\text{\ \ \ \ \ \ \ \ \ \ \
\ \ \ \ \ \ \ \ \ }n\geqslant 0 \\
\widetilde{J}_{-\left( n+1\right) }^{\left( \frac{1}{\widetilde{q}},\frac{r}{%
s}\right) } &=&-\frac{1}{rs}\frac{1}{r^{n}}+\frac{1}{s}\widetilde{J}%
_{-n}^{\left( \frac{1}{\widetilde{q}},\frac{r}{s}\right) }\text{\ \ \ \ \ \
\ \ \ \ \ \ \ \ \ \ \ \ }n\geqslant 0
\end{eqnarray}%
\begin{equation}
\widetilde{J}_{-n}^{\left( \frac{1}{\widetilde{q}},\frac{r}{s}\right)
}=\left\{ 
\begin{array}{c}
0\text{ \ \ \ \ \ \ \ \ \ \ \ \ \ \ \ \ \ \ \ \ \ \ \ \ \ \ \ \ \ \ \ \ \ }%
n=0 \\ 
-\frac{1}{rs}\text{\ \ \ \ \ \ \ \ \ \ \ \ \ \ \ \ \ \ \ \ \ \ \ \ \ \ \ \ \
\ \ }n=1 \\ 
\frac{s+r}{rs}\widetilde{J}_{-\left( n-1\right) }^{\left( \frac{1}{%
\widetilde{q}},\frac{r}{s}\right) }-\frac{1}{rs}\widetilde{J}_{-\left(
n-2\right) }^{\left( \frac{1}{\widetilde{q}},\frac{r}{s}\right) }\text{ \ \
\ \ \ \ \ }n>1%
\end{array}%
\right.
\end{equation}%
The proof of the above recurrence relations follow by direct substitution of
equations (\ref{eqmdualtnegmdu}) in (\ref{eqdualjn2rec}) and (\ref%
{eqdualjn3rec}).

For $q=s-r=3-1$, using equations (\ref{eqdeofnthJacNumberDual}) and (\ref%
{eqdeofnthJacNumberNegativeDual}), we obtain the complete set of the
Jacobsthal-like numbers appearing in the partial sums of the expansion of $%
\left( 1/2\right) $ in powers of $\left( 1/3\right) $ 
\begin{equation}
...,-\frac{364}{729},-\frac{121}{243},-\frac{40}{81},-\frac{13}{27},-\frac{4%
}{9},-\frac{1}{3},0,1,4,13,40,121,364,...
\end{equation}

\section{One more expansion}

\bigskip

Consider the expansion of $1/5$ in powers of $(1/2)$. \ Choosing $w=1:1$,
our method yields the infinite series%
\begin{eqnarray}
\frac{1}{5} &=&X^{\left( 1/5\right) }=0+\left( \frac{1}{2}\right) -\left( 
\frac{1}{2}\right) ^{2}-\left( \frac{1}{2}\right) ^{3}+\left( \frac{1}{2}%
\right) ^{4}+\left( \frac{1}{2}\right) ^{5}-\left( \frac{1}{2}\right) ^{6} 
\notag \\
&&-\left( \frac{1}{2}\right) ^{7}+\left( \frac{1}{2}\right) ^{8}+\left( 
\frac{1}{2}\right) ^{9}-\left( \frac{1}{2}\right) ^{10}-\left( \frac{1}{2}%
\right) ^{11}+\left( \frac{1}{2}\right) ^{12}+...  \label{eqninfser1/5}
\end{eqnarray}%
whose partial sums $X_{n}^{\left( 1/5,1/2\right) }$ are given by%
\begin{eqnarray*}
&&0,\frac{1}{2},\frac{1}{4},\frac{1}{8},\frac{3}{16},\frac{7}{32},\frac{13}{%
64},\frac{25}{128},\frac{51}{256},\frac{103}{512},\frac{205}{1024},\frac{409%
}{2048},\frac{819}{4096} \\
&&,\frac{1639}{8192},\allowbreak \frac{3277}{16\,384},\frac{6553}{32\,768},%
\frac{13\,107}{65\,536},\frac{26\,215}{131\,072},...
\end{eqnarray*}%
It is now a simple matter to show that the partial sums $X_{n}^{\left(
1/5,1/2\right) }$ for this series are given by 
\begin{eqnarray}
X_{2n}^{\left( 1/5,1/2\right) } &=&\frac{1}{4^{n}}\frac{4^{n}-\left(
-1\right) ^{n}}{4+1},\text{ \ \ \ \ \ \ \ \ \ \ }n=0,1,2,3,...
\label{A015521Found} \\
X_{2n+1}^{\left( 1/5,1/2\right) } &=&\frac{1}{2\cdot 4^{n}}\frac{2\cdot
4^{n}+3\left( -1\right) ^{n}}{4+1},\text{ \ \ \ }n=0,1,2,3,...
\label{eqA102900}
\end{eqnarray}%
The sequences in equations (\ref{A015521Found}) and (\ref{eqA102900}) are
precisely the closed form for the sequences: Sloane \textbf{A015521 }and
Sloane\textbf{\ A102900}, respectively\textbf{.} The reader should recall
that the sequence \textbf{A015521 }was already obtained before by expanding $%
1/5$ in powers of $\left( 1/4\right) $, which of course can be obtained by
simply rearranging the terms in the convergent series for the expansion for $%
1/5$ in powers of $\left( 1/2\right) $ in (\ref{eqninfser1/5}) in the
following way%
\begin{eqnarray*}
\frac{1}{5} &=&X^{\left( 1/5\right) }=0+\left[ \left( \frac{1}{2}\right)
-\left( \frac{1}{2}\right) ^{2}\right] \left[ 1-\left( \frac{1}{2}\right)
^{2}+\left( \frac{1}{2}\right) ^{4}-\left( \frac{1}{2}\right) ^{6}+...\right]
\\
&=&\sum\limits_{n=1}^{\infty }\left( -1\right) ^{n+1}\left( \frac{1}{4}%
\right) ^{n}
\end{eqnarray*}%
which is the infinite series one obtained in (\ref{eqex1/5inpo1/4}). Yet a
different arrangement of the terms in the infinite terms in (\ref%
{eqninfser1/5}) leads to%
\begin{eqnarray*}
\frac{1}{5} &=&X^{\left( 1/5\right) }=0+\left[ \left( \frac{1}{2}\right)
-\left( \frac{1}{2}\right) ^{2}-\left( \frac{1}{2}\right) ^{3}+\left( \frac{1%
}{2}\right) ^{4}\right] \left[ 1+\left( \frac{1}{2}\right) ^{4}+\left( \frac{%
1}{2}\right) ^{8}+...\right] \\
&=&\allowbreak \sum\limits_{n=1}^{\infty }3\left( \frac{1}{16}\right) ^{n}
\end{eqnarray*}%
which is an expansion in powers of $\left( 1/16\right) $ whose $n$th partial
sum reads

\begin{equation}
X_{n}^{\left( 1/5\right) }=\frac{1}{5}-\frac{1}{5}\left( \frac{1}{16}\right)
^{n}=\frac{3}{\left( 16\right) ^{n}}\frac{\left( 16\right) ^{n}-1}{16-1},%
\text{ \ \ \ \ \ \ \ \ \ \ }n>0
\end{equation}%
The sequence defined by 
\begin{equation}
S_{n}\equiv \frac{\left( 16\right) ^{n}}{3}X_{n}^{\left( 1/5\right) }=\frac{%
\left( 16\right) ^{n}-1}{16-1},\text{ \ \ \ \ \ \ \ \ \ \ }n>0
\end{equation}%
is Sloane \textbf{A131865, }whose members are 
\begin{equation}
1,17,273,4369,69905,1118481,17895697,286331153,4581298449....
\end{equation}%
In this case $S_{n}$ satisfies the recursion relation%
\begin{eqnarray}
S_{1} &=&1  \notag \\
S_{n} &=&16S_{n-1}+1\text{\ \ \ \ \ \ \ \ \ \ }n>1
\end{eqnarray}%
The most interesting thing here is the form of the partial sums in (\ref%
{eqA102900}). Their form is different from the more-generalized $q-$%
Jacobsthal numbers and their $q-$Jacobsthal-like numbers. This form suggest
expanding our definition to a form where all the expansions considered so
far are just special cases. Hence,

\begin{definition}
For any $a,b$ complex and any nonnegative integer $n$, the $\mathcal{A}$
numbers are defined by
\end{definition}

\begin{equation}
\mathcal{A}_{n}^{\left\{ a,b,s,t\right\} }=\frac{a\cdot s^{n}+\left(
-1\right) ^{n}b\cdot t^{n}}{s+t},\text{ \ \ \ }n=0,1,2,3,...
\label{eqnMostambformula}
\end{equation}%
This formula reduces to the following: (a) to the Jacobsthal formula for $%
a=-b=t=1$ and $s=2$, (b) to the generalization obtained in ref. \cite{Ji
Young Choi} for $a=-b=1$ and $t=1$, (c) to the more-generalized $q-$%
Jacobsthal numbers for $a=-b=1$, $t=r$ and (d) to the more-generalized $q-$%
Jacobsthal-like numbers for $a=-b=1$ and $t=-r$.

\section{\protect\bigskip Continuation of $J$, $\ \protect\widetilde{J}$ and 
$\mathcal{A}$ to complex numbers}

So far in our derivation of the more generalized Jacobsthal numbers and
their Jacobsthal-like numbers we have assumed integer values for $r$ and $s$
in equations (\ref{eqmoregejacnumbr/r+s}) and (\ref{eqdeofnthJacNumberDual}) 
\begin{equation}
J_{n}^{\left( \frac{1}{q},\frac{r}{s}\right) }=\frac{\left( s\right)
^{n}-\left( -r\right) ^{n}}{s+r}
\end{equation}%
and%
\begin{equation}
\widetilde{J}_{n}^{\left( \frac{1}{\widetilde{q}},\frac{r}{s}\right) }=\frac{%
\left( s\right) ^{n}-\left( r\right) ^{n}}{s-r}
\end{equation}%
We can achieve continuation by simply assuming the validity of the above
equations for $r$ and $s$ complex (or by assuming that their recurrence
equations provide the continuation to complex values of $r,s$). Thus let $%
r=\mu $ and $s=\nu $, where now $\mu $ and $\nu $ can assume any complex
values. Thus our defining equations above may be written as 
\begin{equation}
J_{n}^{\left( \frac{1}{q},\frac{\mu }{\nu }\right) }=\frac{\left( \nu
\right) ^{n}-\left( -\mu \right) ^{n}}{\nu +\mu }  \label{eqgejsmore}
\end{equation}%
and%
\begin{equation}
\widetilde{J}_{n}^{\left( \frac{1}{\widetilde{q}},\frac{\mu }{\nu }\right) }=%
\frac{\left( \nu \right) ^{n}-\left( \mu \right) ^{n}}{\nu -\mu }
\label{eqgejs-like-more}
\end{equation}%
where $q=\mu +\nu $\textbf{.} Moreover, we can continue the value of $n$ to
all complex numbers by assuming that the above equations hold for $n$
complex or alternatively by assuming that their recurrence equations provide
the continuation to complex values of $n$. Thus letting $n=\lambda $, where $%
\lambda $ is now any complex number, and furthermore relaxing the assumption
that $q=\mu +\nu $, we define $j_{\lambda }^{\left( \mu ,\nu \right) }$ by 
\begin{equation}
j_{\lambda }^{\left( \mu ,\nu \right) }=\frac{\left( \nu \right) ^{\lambda
}-\left( \mu \right) ^{\lambda }}{\nu -\mu }  \label{eqcontJS}
\end{equation}%
We have seen before that for $2\nu =1+\sqrt{5},$ $2\mu =1-\sqrt{5}$ and $%
\lambda =n$ (nonnegative integer), equation (\ref{eqcontJS}) yields the 
\textbf{Fibonacci }sequence. For $\lambda =n$ (nonnegative integer), $\nu =1+%
\sqrt{2},$ $\mu =1-\sqrt{2}$, equation (\ref{eqcontJS}) yields\textbf{\ Pell 
}sequences. If we set $\nu =1/2$, $\mu =1/3$ and $\lambda =n$, the above
formula yields the following sequence%
\begin{equation}
0,1,5,19,65,211,665,2059,6305,19171,58025,175099,...
\end{equation}%
which is Sloane \textbf{A001047}. For $\lambda =n$ (nonnegative integer), $%
\nu =1+\sqrt{3},$ $\mu =1-\sqrt{3}$, equation (\ref{eqcontJS}) yields Sloane 
\textbf{A002605}:%
\begin{equation}
0,1,2,6,16,44,120,328,896,2448,6688,18272,49920,...
\end{equation}%
This form may be now used to compute $j_{\lambda }^{\left( \mu ,\nu \right)
} $ for complex values of $\mu $, $\nu $ and $\lambda $ and see if the
resulting numbers are of any interest.

We conclude this section by analytically continuing the formula in (\ref%
{eqnMostambformula}) into the complex plane. This we do by asserting that 
\begin{equation}
\mathcal{A}_{\lambda }^{\left\{ a,b,s,t\right\} }=\frac{a\cdot \nu ^{\lambda
}+\left( -1\right) ^{\lambda }b\cdot \gamma ^{\lambda }}{\nu +\gamma },
\label{eqnMostambformulacont}
\end{equation}%
where now $a,b,\nu ,\gamma ,\lambda $ may assume any complex values, provide
the analytic continuation of (\ref{eqnMostambformula}) when $s,t$, and $n$
assume complex values. Since (\ref{eqnMostambformulacont}) reduces to (\ref%
{eqnMostambformula}) for $\nu ,\gamma ,\lambda $ integers then by the
identity theorem of two functions it is the analytic continuation of (\ref%
{eqnMostambformula}).

\section{Extending the method to all numbers}

In fact our method can be used to expand all numbers, real (or complex) in
powers of $r/s$ with real (or complex) coefficients. For example the
irrational number $1/\pi $ can be expanded in powers of $r/s=1/2$ by simply
putting $\pi -1$ units of mass at the origin and $1$ unit of mass at the $%
x=1 $. Then the formula in (\ref{eqnhscoordn}) for the center of mass yields%
\begin{equation}
X=\frac{\left( \pi -1\right) \times 0+1\times 1}{\left( \pi -1\right) +1}=%
\frac{1}{\pi }
\end{equation}%
which is the desired irrational number to expand. Choosing $w=1:1$, our
iterative method results in an infinite series in powers of $1/2$ 
\begin{eqnarray}
\frac{1}{\pi } &=&X^{\left( 1/\pi \right) }=0+\left( \frac{1}{2}\right)
-\left( \frac{1}{2}\right) ^{2}+\left( \frac{1}{2}\right) ^{3}-\left( \frac{1%
}{2}\right) ^{4}+\left( \frac{1}{2}\right) ^{5}-\left( \frac{1}{2}\right)
^{6}  \notag \\
&&-\left( \frac{1}{2}\right) ^{7}-\left( \frac{1}{2}\right) ^{8}+\left( 
\frac{1}{2}\right) ^{9}-\left( \frac{1}{2}\right) ^{10}+\left( \frac{1}{2}%
\right) ^{11}+\left( \frac{1}{2}\right) ^{12}  \notag \\
&&+\left( \frac{1}{2}\right) ^{13}+\left( \frac{1}{2}\right) ^{14}+\left( 
\frac{1}{2}\right) ^{15}+\cdot \cdot
\end{eqnarray}%
It is not hard to see that this is very fast convergent series\footnote{%
The easiest way to check that the above series converges to $1/\pi $ \ very
fast and obtain the desired level of accuracy to the true value of $\pi $ is
to write a computer program to evaluate the partial sums using the CM method.%
}. Let us compute the first few partial sums $X_{n}^{\left( 1/\pi
,1/2\right) }$. For $n=0,1,2,...15,...$, we have%
\begin{eqnarray}
X_{n}^{\left( 1/\pi ,1/2\right) } &=&0,\frac{1}{2},\frac{1}{4},\frac{3}{8},%
\frac{5}{16},\frac{11}{32},\frac{21}{64},\frac{41}{128},\frac{81}{256},\frac{%
163}{512},  \notag \\
&&\frac{325}{1024},\frac{651}{2048},\frac{1303}{4096},\frac{2607}{8192},%
\frac{5215}{16\,384},\frac{10\,431}{32\,768},...
\end{eqnarray}%
Now if we compute the difference $X_{n}^{\left( 1/\pi ,1/2\right) }-1/\pi $,
we see that it tends to zero very fast%
\begin{eqnarray*}
X_{0}^{\left( 1/\pi ,1/2\right) }-\frac{1}{\pi } &\simeq &-0.3183 \\
X_{1}^{\left( 1/\pi ,1/2\right) }-\frac{1}{\pi } &\simeq &\allowbreak
0.181\,69 \\
X_{2}^{\left( 1/\pi ,1/2\right) }-\frac{1}{\pi } &\simeq &-0.0683 \\
X_{4}^{\left( 1/\pi ,1/2\right) }-\frac{1}{\pi } &\simeq &-0.0058 \\
X_{15}^{\left( 1/\pi ,1/2\right) }-\frac{1}{\pi } &\simeq &0.00001
\end{eqnarray*}%
Unfortunately, I have not managed to get a closed form for the $n$th partial
sum, $X_{n}^{\left( 1/\pi ,1/2\right) }$, and I doubt if any exist.

Likewise, the expansion of $\pi $ can be obtained by considering for example 
$q$ units of mass (with $q>\pi $) and dividing $q$ into two groups, $q-\pi $
units of mass and $\pi $ units of mass, and choosing the initial
configurations such that the $q-\pi $ units of mass are placed at $x=0$ and
the $\pi $ units of mass are placed at $x=1$. In this case our formula for
the center of mass in (\ref{eqnhscoordn}) yields%
\begin{equation}
X^{\left( \pi /q\right) }=\frac{\left( q-\pi \right) \times 0+\pi \times 1}{%
\left( q-\pi \right) +\pi }=\frac{\pi }{q}
\end{equation}%
or%
\begin{equation}
\pi =qX^{\left( \pi /q\right) }
\end{equation}%
Since our method always yields an infinite series for $X^{\left( \pi
/q\right) }$ (the center of mass) in powers of $\left( 1/2\right) $; our
method is guaranteed to give an expansion of $\pi $ in powers of $\left(
1/2\right) $. An expansion of $\pi $ is given in the appendix C.

We clearly established that our iterative method for computing the center of
mass for $q$ units of mass on the unit interval provides us with a definite
procedure for expanding all numbers in powers of $r/s$, with $r,s$ integers
and $s>r>0$\footnote{%
In fact $r/s$ need not be a rational number. However, the procedure provided
by the CM method is then more involved.}. The results we have obtained so
far follow from the following theorem:

\begin{theorem}
Every number real (or complex) has an infinite series expansion in powers of 
$r/s$\ with real (or complex) coefficients, for some rational number $r/s$%
\thinspace $<1$. The proof of the theorem follows at once from the CM method.
\end{theorem}

\section{Conclusions}

\bigskip We have presented a way to expand rational numbers as infinite
series in powers of rational numbers less than unity and have showed that
partial sums associated with various expansions are related to special
numbers appearing in mathematics such as Jacobsthall numbers and the
generalized Fibonacci sequence and provided a generalization to the
Jacobsthall numbers. The generalization of the Jacobsthall numbers lead us
to what we called the "more generalized $q-$Jacobsthal numbers" and the
"more-generalized $q-$ Jacobsthall-like numbers". Moreover, we analytically
continued our results to negative integers which provided a completeness
that was missing in the usual theory of Jacobsthall numbers. We concluded
the first part of the paper by analytically continued all the parameters of
our results to the entire complex plane. The analytically continuing results
gave rise to Fibonacci sequence, Pell numbers, among other sequences of
interest to mathematicians and scientists. We concluded our paper by giving
an expansion of $\pi $ in powers of $\left( 1/2\right) $ and noted that the
series gives a good approximation of $\pi $ with only the first few terms
considered. Our systems considered here exhibit the behavior of a quantized
harmonic oscillator a topic that is addressed in a sequel to this paper.

\bigskip

\appendix

\bigskip

\section{\protect\bigskip The extension of known identities to our numbers}

Here we present some of the identity satisfied by our numbers and outline
their proofs.

\begin{theorem}
(generalized Catalan'sidentity) For any nonnegative integers $s,r$ (with $%
s>r $), $m$ and $n$ (with $n>m$), we have%
\begin{equation}
\widetilde{J}_{n-m}^{\left( \frac{1}{\widetilde{q}},\frac{r}{s}\right) }%
\widetilde{J}_{n+m}^{\left( \frac{1}{\widetilde{q}},\frac{r}{s}\right) }-%
\widetilde{J}_{n}^{\left( \frac{1}{\widetilde{q}},\frac{r}{s}\right)
}=-\left( sr\right) ^{n-m}\widetilde{J}_{m}^{\left( \frac{1}{\widetilde{q}},%
\frac{r}{s}\right) }
\end{equation}
\end{theorem}

To see the truth of the above theorem, we substitute (\ref%
{eqdeofnthJacNumberDual}) into the LHS of the above equation%
\begin{eqnarray*}
&&\widetilde{J}_{n-m}^{\left( \frac{1}{\widetilde{q}},\frac{r}{s}\right) }%
\widetilde{J}_{n+m}^{\left( \frac{1}{\widetilde{q}},\frac{r}{s}\right) }-%
\widetilde{J}_{n}^{\left( \frac{1}{\widetilde{q}},\frac{r}{s}\right) } \\
&=&\left( \frac{s^{n-m}-r^{n-m}}{s-r}\right) \left( \frac{s^{n+m}-r^{n+m}}{%
s-r}\right) -\left( \frac{s^{n}-r^{n}}{s-r}\right) ^{2} \\
&=&-\left( sr\right) ^{n-m}\left( \frac{s^{m}-r^{m}}{s-r}\right)
^{2}=-\left( sr\right) ^{n-m}\widetilde{J}_{m}^{\left( \frac{1}{\widetilde{q}%
},\frac{r}{s}\right) }
\end{eqnarray*}%
Now the generalized Catalan's identity for the $n-$th more-generalized $q-$%
Jacobsthall numbers, $J_{n}^{\left( \frac{1}{q},\frac{r}{s}\right) }$,
follows from the above identity with the replacement of $r$ by $-r$. Thus we
have the following corollary.

\begin{corollary}
(generalized Catalan's identity) For any nonnegative integers $s,r$ (with $%
s>r$)$,m$ and $n$ (with $n>m$), we have%
\begin{equation}
J_{n-m}^{\left( \frac{1}{q},\frac{r}{s}\right) }J_{n+m}^{\left( \frac{1}{q},%
\frac{r}{s}\right) }-J_{n}^{\left( \frac{1}{q},\frac{r}{s}\right) }=-\left(
-sr\right) ^{n-l}J_{m}^{\left( \frac{1}{q},\frac{r}{s}\right) }
\end{equation}
\end{corollary}

The next theorem we prove is the convolution identity.

\begin{theorem}
(generalized convolution identity) For any nonnegative integers $s,r$ (with $%
s>r$),$n$ and $m$, we have
\end{theorem}

\begin{equation}
\widetilde{J}_{n+m}^{\left( \frac{1}{\widetilde{q}},\frac{r}{s}\right) }=%
\widetilde{J}_{n+1}^{\left( \frac{1}{\widetilde{q}},\frac{r}{s}\right) }%
\widetilde{J}_{m}^{\left( \frac{1}{\widetilde{q}},\frac{r}{s}\right) }-sr%
\widetilde{J}_{n}^{\left( \frac{1}{\widetilde{q}},\frac{r}{s}\right) }%
\widetilde{J}_{m-1}^{\left( \frac{1}{\widetilde{q}},\frac{r}{s}\right) }
\end{equation}%
To prove this identity consider%
\begin{eqnarray*}
&&\widetilde{J}_{n+1}^{\left( \frac{1}{\widetilde{q}},\frac{r}{s}\right) }%
\widetilde{J}_{m}^{\left( \frac{1}{\widetilde{q}},\frac{r}{s}\right) }-sr%
\widetilde{J}_{n}^{\left( \frac{1}{\widetilde{q}},\frac{r}{s}\right) }%
\widetilde{J}_{m-1}^{\left( \frac{1}{\widetilde{q}},\frac{r}{s}\right) } \\
&=&\left( \frac{s^{n+1}-r^{n+1}}{s-r}\right) \left( \frac{s^{m}-r^{m}}{s-r}%
\right) -sr\left( \frac{s^{n}-r^{n}}{s-r}\right) \left( \frac{s^{m-1}-r^{m-1}%
}{s-r}\right) \\
&=&\left( \frac{s^{n+m}-r^{n+m}}{s-r}\right) =\widetilde{J}_{n+m}^{\left( 
\frac{1}{\widetilde{q}},\frac{r}{s}\right) }
\end{eqnarray*}%
The convolution identity for the $n-$th more generalized $q-$Jacobsthall
numbers, $J_{n}^{\left( \frac{1}{q},\frac{r}{s}\right) }$, follows from the
above identity with the replacement of $r$ by $-r$. Thus we have the
following corollary (convolution identity).

\begin{corollary}
(convolution identity) For any nonnegative integers $s,r$ (with $s>r$), $n$
and $m$, we have%
\begin{equation}
J_{n+m}^{\left( \frac{1}{q},\frac{r}{s}\right) }=J_{n+1}^{\left( \frac{1}{q},%
\frac{r}{s}\right) }J_{m}^{\left( \frac{1}{q},\frac{r}{s}\right)
}+srJ_{n}^{\left( \frac{1}{q},\frac{r}{s}\right) }J_{m-1}^{\left( \frac{1}{q}%
,\frac{r}{s}\right) }
\end{equation}
\end{corollary}

Next we prove the D'Ocagne's identity.

\begin{theorem}
(generalized D'Ocagne's identity) For any nonnegative integers $s,r$ (with $%
s>r$), $m$ and $n$ (with $n>m$), we have
\end{theorem}

\begin{equation}
\widetilde{J}_{n}^{\left( \frac{1}{\widetilde{q}},\frac{r}{s}\right) }%
\widetilde{J}_{m+1}^{\left( \frac{1}{\widetilde{q}},\frac{r}{s}\right) }-%
\widetilde{J}_{n+1}^{\left( \frac{1}{\widetilde{q}},\frac{r}{s}\right) }%
\widetilde{J}_{m}^{\left( \frac{1}{\widetilde{q}},\frac{r}{s}\right)
}=\left( sr\right) ^{m}\widetilde{J}_{n-m}^{\left( \frac{1}{\widetilde{q}},%
\frac{r}{s}\right) }
\end{equation}%
To prove the above theorem consider%
\begin{eqnarray*}
&&\widetilde{J}_{n}^{\left( \frac{1}{\widetilde{q}},\frac{r}{s}\right) }%
\widetilde{J}_{m+1}^{\left( \frac{1}{\widetilde{q}},\frac{r}{s}\right) }-%
\widetilde{J}_{n+1}^{\left( \frac{1}{\widetilde{q}},\frac{r}{s}\right) }%
\widetilde{J}_{m}^{\left( \frac{1}{\widetilde{q}},\frac{r}{s}\right) } \\
&=&\left( \frac{s^{n}-r^{n}}{s-r}\right) \left( \frac{s^{m+1}-r^{m+1}}{s-r}%
\right) -\left( \frac{s^{n+1}-r^{n+1}}{s-r}\right) \left( \frac{s^{m}-r^{m}}{%
s-r}\right) \\
&=&\left( sr\right) ^{m}\left( \frac{s^{n-m}-r^{n-m}}{s-r}\right) =\left(
sr\right) ^{m}\widetilde{J}_{n-m}^{\left( \frac{1}{\widetilde{q}},\frac{r}{s}%
\right) }
\end{eqnarray*}%
From the above identity, we can find the generalized D'Ocagne's identity for
the $n$th more-generalized $q-$Jacobsthall numbers, $J_{n}^{\left( \frac{1}{q%
},\frac{r}{s}\right) }$, with replacing $r$ by $-r$. Thus we have the
following corollary.

\begin{corollary}
(generalized D'Ocagne's identity) For any nonnegative integers $s,r$ (with $%
s>r$), $m$ and $n$ (with $n>m$), we have%
\begin{equation}
J_{n}^{\left( \frac{1}{q},\frac{r}{s}\right) }J_{m+1}^{\left( \frac{1}{q},%
\frac{r}{s}\right) }-J_{n+1}^{\left( \frac{1}{q},\frac{r}{s}\right)
}J_{m}^{\left( \frac{1}{q},\frac{r}{s}\right) }=\left( -1\right) ^{m}\left(
sr\right) ^{m}J_{n-m}^{\left( \frac{1}{q},\frac{r}{s}\right) }
\end{equation}
\end{corollary}

\section{Expansion of $\protect\pi $\protect\bigskip\ }

If we choose $q-\pi =4-\pi $, and follow our method with the zero
approximation being the position of the larger collection of masses ($\pi $
units of mass), we obtain the following series%
\begin{eqnarray*}
\frac{\pi }{4} &=&X^{\left( \pi /4\right) }=1-\left( \frac{1}{2}\right)
+\left( \frac{1}{2}\right) ^{2}+\left( \frac{1}{2}\right) ^{3}-\left( \frac{1%
}{2}\right) ^{4}-\left( \frac{1}{2}\right) ^{5}+\left( \frac{1}{2}\right)
^{6} \\
&&-\left( \frac{1}{2}\right) ^{7}-\left( \frac{1}{2}\right) ^{8}+\left( 
\frac{1}{2}\right) ^{9}-\left( \frac{1}{2}\right) ^{10}-\left( \frac{1}{2}%
\right) ^{11}-\left( \frac{1}{2}\right) ^{12} \\
&&-\left( \frac{1}{2}\right) ^{13}+\left( \frac{1}{2}\right) ^{14}+\left( 
\frac{1}{2}\right) ^{15}+\left( \frac{1}{2}\right) ^{16}+\left( \frac{1}{2}%
\right) ^{17}+....
\end{eqnarray*}%
This series converse very fast. For example for $n$ small (say $n=17,18,19$%
), the difference between $\frac{\pi }{4}$ and the partial sums $%
X_{n}^{\left( \pi /4\right) }$is 
\begin{eqnarray*}
\frac{\pi }{4}-X_{17}^{\left( \pi /4\right) } &\simeq &5.\,\allowbreak
402\,2\times 10^{-6} \\
\frac{\pi }{4}-X_{18}^{\left( \pi /4\right) } &\simeq &1.\,\allowbreak
587\,5\times 10^{-6} \\
\frac{\pi }{4}-X_{19}^{\left( \pi /4\right) } &\simeq &-3.\,\allowbreak
198\,8\times 10^{-7}
\end{eqnarray*}%
which is quite remarkable. Thus our method yields expansion for $\pi $ in
powers of $\left( r/s\right) $ , with $r,s$, integers and $s>r>0$.

\bigskip

\bigskip

\end{document}